\documentclass[reqno,a4paper,12pt]{amsart}

\usepackage[all,poly]{xy}
\usepackage{amsfonts,stmaryrd}
\usepackage[mathcal]{eucal}
\usepackage{amssymb}
\usepackage{amsmath}
\usepackage{mathrsfs}
\usepackage{enumerate}

\def\A{\operatorname{A}}

\def\E{\operatorname{E}}

\def\F{\operatorname{F}}
\def\K{\operatorname{K}}

\def\map{\longrightarrow}

\def\GL{\operatorname{GL}}

\def\SL{\operatorname{SL}}

\def\Sp{\operatorname{Sp}}
\def\SK{\operatorname{SK}}

\def\EO{\operatorname{EO}}

\def\Cent{\operatorname{Cent}}

\def\sr{\operatorname{sr}}

\def\Rat{{\mathbb Q}}

\def\Int{{\mathbb Z}}

\def\pamod#1{\,(\operatorname{mod}{\, #1})\,}

\newtheorem{theorem}{Theorem}

\newtheorem{lemma}{Lemma}
\newtheorem{problem}{Problem}

\title{Relative centralisers of relative subgroups}

\author{Nikolai Vavilov}
\address{Department of Mathematics and Mechanics,\\
St.~Petersburg State University,\\ St.~Petersburg, Russia}
\email{nikolai-vavilov@yandex.ru}
\thanks{This publication is supported by Russian Science Foundation grant 17-11-01261.}

\author{Zuhong Zhang}
\address{Department of  Mathematics\\
 Beijing Institute of Technology\\
 Beijing, China}
\email{zuhong@hotmail.com}


\begin{document}


\begin{abstract}
Let $R$ be an associative ring with 1, $G=\GL(n, R)$ be the general 
linear group of degree $n\ge 3$ over $R$. 
In this paper we calculate the relative centralisers of the relative
elementary subgroups or the principal congruence subgroups, corresponding
to an ideal $A\unlhd R$ modulo the relative elementary subgroups or 
the principal congruence subgroups, corresponding to another ideal 
$B\unlhd R$. Modulo congruence subgroups the results are essentially
easy exercises in linear algebra. But modulo the elementary subgroups
they turned out to be quite tricky, and we could get definitive answers
only over commutative rings, or, in some cases, only over Dedekind rings.
We discuss also some further related problems, such as the 
interrelations of various birelative commutator subgroups,  etc., 
and state several unsolved questions.
\end{abstract}

\maketitle


\section{Introduction}

Let $F,H\le G$ be two subgroups of $G$. We consider the centraliser 
of $F$ modulo $H$
$$ C_{G}(F,H)=\big\{ g\in G \mid \forall f\in F, [f,g]\in H\big\}. $$
\noindent
If $H\unlhd G$ is a normal subgroup, and $\pi_H:G\map G/H$ is
the corresponding projection, then 
$$ C_{G}(F,H)=\pi_H^{-1}(C_{G/H}(FH/H)). $$
\noindent
is the preimage of the corresponding absolute centraliser in the 
factor-group $G/H$.
\par
In the present paper, we are interested in the case of the general 
linear group $G=\GL(n, R)$ of degree $n\ge 3$ over an associative 
ring $R$ with 1. 
In connection with our project on subgroups normalised by unrelative 
and relative elementary subgroups $E(n,J)$ and $E(n,R,J)$ (see
\cite{NV19rr} for further references), we had to compute the 
centralisers of relative subgroups modulo some other relative subgroups.
\par
We could not find the corresponding results in the available literature.
In fact, in many cases such similar centralisers were extensively studied,
starting with Bass' classical results on the structure of $\GL(n,R)$
in the stable range \cite{Bass_stable}. However, Hyman Bass himself 
and his followers only considered the absolute case, where the 
subgroups $F$ was either perfect itself, or contained a perfect subgroup 
of the same level.
\par
Zenon Borewicz, the first author, and their schools, have  performed 
diverse calculations in this spirit, in connection with structure theory, 
and description of various classes of intermediate subgroups, see, for
instance, \cite{borvav,VPEp,VPEO,VG,VN,VLE6,VLE7,AVS}, etc. However,
in most of these calculations the subgroup $F$ was perfect as well.
\par
Here, we calculate some of these centralisers for the case where 
$G=\GL(n,R)$, whereas {\it both\/} $F$ and $H$ are various relative subgroups 
of $G$, corresponding to proper ideals, such as $E(n,I)$, $E(n,R,I)$, 
$\GL(n,R,I)$, or the like. None of these groups is anywhere close to
being perfect, so that our calculations here are quite different in spirit 
from the calculations in the above papers.
\par
For the case of congruence subgroups 
the corresponding results are mostly exercises in linear algebra, and
hold over arbitrary associative rings. But for elementary subgroups the
answers crucially depend on difficult results of commutator calculus developed
in our joint works with Roozbeh Hazrat and Alexei Stepanov, and then
recently by ourselves, and only hold in modified forms, or under 
miscellaneous assumptions.
\par
The paper is organised as follows. In \S~2 we recall some ideal arithmetic, 
and in \S~3 we collect the definitions of various relative subgroups. After that
in \S~4 we prove our first main result, Theorem 1, which calculates
$C_{GL(n,R)}(E(n,A),\GL(n,R,B))$, over arbitrary associative rings.
In \S~5 we recall the requisite facts on generation of relative elementary
subgroups and their commutators. After that, in \S~6 we explore what can be
done in this spirit for $C_{GL(n,R)}(E(n,A),E(n,R,B))$, over commutative
rings, and prove our second main result, Theorem 2. In particular, it gives
the definitive answer for Dedekind rings of arithmetic type, Theorem 3. 
Next, in \S\S~7 and~8 we adress two closely related problems on intersections 
of relative elementary subgroups, and on rewriting a relative commutator of elementary subgroups in terms of 
different ideals. Finally, in \S~9 we state some further unsolved problems.


\section{Ideal quotient}

Let, as above, $R$ be an associative but not necessarily 
commutative ring with 1. The results of this paper heavily
depend on the operations on ideals of $R$. For two-sided ideals
$A,B\unlhd R$ their sum $A+B$, their intersection $A\cap B$, their
products $AB$ and $BA$, their symmetrised product $A\circ B=
AB+BA$, and their commutator $[A,B]$ are again two sided ideals, 
and their properties are classically known. However, in the 
non-commutative case 
we could not find an authoritative source on ideal quotient 
$(B:A)$, so in this section we collect the some basic facts used 
in the sequel.
\par
For two {\it left\/} ideals
$A$ and $B$ in $R$ we consider their {\it right ideal quotient\/}
$$ BA^{-1}=\{x\in R\mid xA\subseteq B\}. $$
\noindent
Obviously, it is a two sided ideal of $R$. Indeed, for any $y,z\in R$ 
one has $(xy)A=x(yA)\le xA\le B$ and $(zx)A=z(xA)\le zB\le B$.
\par
Similarly, for two {\it right\/} ideals $A$ and $B$ their 
{\it left ideal quotient\/}
$$ A^{-1}B=\{x\in R\mid Ax\subseteq B\}, $$
\noindent
is a two sided ideal of $R$. 
\par\smallskip
\noindent
{\bf Warning.} In many texts the {\it right\/} ideal quotient $BA^{-1}$ 
of two left ideals $A$ and $B$ is called {\it left\/} ideal quotient, 
and is denoted by $(B:A)_L$ or $(B:_LA)$, see \cite{Steinfeld}, for
instance. The most amazing notational convention
is adopted in \cite{BGTV}. There, the {\it right\/} ideal quotient
 $BA^{-1}$ is denoted by ${}_R(B:A)$ --- and is still called
 {\it left\/} ideal quotient. Our notation follows that of \cite{Murdoch,Marubayashi}.
\par\smallskip
Now, for two sided ideals $A$ and $B$ of $R$ their {\it ideal 
quotient\/} is defined as
$$ (B:A)=BA^{-1}\cap A^{-1}B=
\{x\in R\mid xA,Ax\subseteq B\}, $$
\noindent
Clearly, $(B:A)$ is a two sided ideal such that $(B:A)\ge B$, 
and $A\le B$ implies that $(B:A)=R$. In particular $(A:A)=R$ 
and $(A:R)=A$. 
For commutative rings
$(B:A)=BA^{-1}=A^{-1}B$ coincides with the usual ideal
quotient in commutative algebra. 
\par
Let us list some obvious properties of the ideal quotient.
\par\smallskip
$\bullet$ Clearly,
$$ (B:A)\circ A=A(B:A)+(B:A)A\le A(A^{-1}B)+(BA^{-1})A\le B, $$ 
\noindent
thus, $(B:A)$ can be defined as the largest two sided ideal 
$C\unlhd R$ such that $C\circ A\le B$. However, only very rarely
this inclusion becomes an equality.
\par\smallskip
\noindent
{\bf Warning.} The ideal quotient is not a fractional ideal, and
even when $A\ge B$ the ideal quotient $(B:A)$ should not be
interpreted as the {\it fraction\/} of $B$ by $A$. In fact, among
commutative domains the equality $A(B:A)=B$ characterises
Dedekind domains, in this case $(B:A)$ is indeed $BA^{-1}$
in the group of fractional ideals of $R$. The same condition 
imposed on finitely generated ideals characterises Prüfer domains.
\par\smallskip
$\bullet$ 
$(A:(B+C))=(A:B)\cap (A:C)$.
\par\smallskip
Clearly, this equality implies that $(A:(A+B))=(A:A)\cap (A:B)=(A:B)$.
In other words, every ideal quotient $(A:B)$ coincides with such an ideal 
quotient that $A\le B$.
\par\smallskip
$\bullet$ 
$((A\cap B):C)=(A:C)\cap (B:C)$.
\par\smallskip
$\bullet$ Intersection of any two of the ideals 
$$ ((A:B):C),\qquad (A:(B\circ C)),\qquad ((A:C):B) $$ 
\noindent
is contained in the third one.
In particular, when $R$ is commutative, one has 
$$ ((A:B):C)=(A:(BC))=((A:C):B). $$
\par\smallskip
$\bullet$ 
$((A+B):C)\ge (A:C)+(B:C)$.
\par\smallskip
$\bullet$ 
$(A:(B\cap C))\ge (A:B)+(A:C)$.
\par\smallskip
\noindent
{\bf Warning.} Only very rarely these last inequalities become
equalities. Again, among commutative domains any of the equalities
$((A+B):C)=(A:C)+(B:C)$ or $(A:(B\cap C))=(A:B)+(A:C)$ characterises
Dedekind domains. Any of these conditions 
imposed on finitely generated ideals characterises Prüfer domains.
\par\smallskip
Now, let $Z\subseteq R$ be a subset. Its centraliser
$$ \Cent_R(Z)=\{x\in R\mid \forall z\in Z, xz-zx=0\} $$ 
\noindent
is a unital subring of $R$ containing the centre $\Cent(R)=\Cent_R(R)$. 
In the next section we also encounter the {\it relative\/} centraliser of 
a subset $Z$ modulo an ideal $B$:
$$ \Cent_R(Z,B)=\{x\in R\mid \forall z\in Z, xz-zx\in B\}. $$ 
\noindent
Clearly, 
$$ \Cent_R(Z,B)=\rho_B^{-1}\big(\Cent_{R/B}\big(\rho_B(Z)\big)\big), $$
\noindent
is a unital subring of $R$ containing both the absolute centraliser 
$\Cent_R(Z)$, and the ideal $B$ itself, $\Cent_R(Z)+B\le \Cent_R(Z,B)$.
\par
Mostly, we consider relative centralisers of ideals. Let $Z=A\unlhd R$. Then 
$\rho_B(A)=(A+B)/B$ and for $x\in(B:A)$ and $a\in A$ one has $xa,ax\in B$, 
so that in fact even
$$ \Cent_R(A)+(B:A)\le\Cent_R(A,B). $$


\section{Relative subgroups}

For two subgroups $F,H\le G$, we denote by $[F,H]$ their mutual 
commutator subgroup spanned by all commutators $[f,h]$, where 
$f\in F$, $h\in H$. Observe that our commutators are always 
left-normed, $[x,y]=xyx^{-1}y^{-1}$. The double commutator 
$[[x,y],z]$ will be denoted simply by $[x,y,z]$. As usual, $C(G)$
denotes the centre of a group $G$, whereas $C_{\GL(n,R)}(H)$ denotes
the centraliser of a subgroup $H\le G$ in $G$. 

\par
 As usual, $e$ 
 denotes the identity matrix and $e_{ij}$ is a standard matrix unit. For  
 $\xi\in R$ and $1\le i\neq j\le n$, we denote by $t_{ij}(\xi)=e+\xi e_{ij}$, 
 we denote the corresponding [elementary] transvection. To any ideal 
 $I\unlhd R$ one associates the elementary subgroup 
 $$ E(n,I)=\big\langle t_{ij}(\xi),\ \xi\in I,\ 1\le i\neq j\le n\big\rangle, $$
 \noindent
generated by all elementary transvections of level $I$, and the 
 \textit{relative} elementary subgroup $E_I=E(n,R,I)$ of level $I$ is 
defined as the normal closure of $E(n, I)$ in the \textit{absolute} 
elementary subgroup $E=E(n,R)$. 

Further, consider the reduction homomorphism 
$\rho_I:\GL(n,R)\longrightarrow\GL(n,R/I)$ modulo $I$. 
\par\smallskip
$\bullet$ 
By definition,
the \textit{principal} congruence subgroup $\GL(n,R,I)$ is the
kernel of $\rho_I$. In other words, $\GL(n,R,I)$ consists of all matrices $g$ 
congruent to $e$ modulo $I$.
$$ \GL(n,R,I)=\big\{ g=(g_{ij})\in\GL(n,R)\mid 
g_{ij}\equiv\delta_{ij}\pamod{I}\big\}. $$
\par\smallskip
$\bullet$ 
In turn, the \textit{full} congruence 
subgroup $C(n,R,I)$ is the full preimage of the center of 
$\GL(n,R/I)$ with respect to $\rho_I$. In other words, $C(n,R,I)$ consists
of matrices, which become scalar modulo $I$, i.e.\ have the form 
$\lambda e$, where $\lambda$ is central modulo $I$, 
$\lambda\in\Cent(R/I)^*$.
\par\smallskip
We need also some of the less familiar congruence subgroups.
\par\smallskip
$\bullet$ 
The {\it brimming\/} congruence subgroup $G(n,R,I)$, which is the
full preimage of the diagonal subgroup $D(n,R/I)\le\GL(n,R/I)$. In the
terminology of Zenon Borewicz, $G(n,R,I)=G(\sigma)$ is the net 
subgroup corresponding to the $D$-net $\sigma=(\sigma_{ij})$, 
$1\le i\neq j\le n$, such that $\sigma_{ij}=I$ for all $i\neq j$, while
$\sigma_{ii}=R$ as they should be, for $D$-nets, see \cite{BV80,borvav}.
$$ G(n,R,I)=\big\{ g=(g_{ij})\in\GL(n,R)\mid 
g_{ij}\equiv 0\pamod{I}, i\neq j\big\}. $$
\par\smallskip
$\bullet$ For a subgroup $\Omega\le (R/I)^*$ we can define
$C_{\Omega}(n,R,I)$ consists of matrices which modulo $I$, 
have the form $\lambda e$, for some $\lambda\in\Omega$.
The largest one of those is the group consisting of all matrices that 
become (non-central!) homotheties modulo $I$, it corresponds to 
$\Omega=(R/I)^*$:
$$ C^*(n,R,I)=\big\{ g=(g_{ij})\in G(n,R,I)\mid 
g_{ii}\equiv g_{jj}\pamod{I}\big\}. $$
\par\smallskip
$\bullet$ But actually, we will be most interested in the following 
special case. Let $A,B\unlhd R$ be two ideals of $R$. We
consider the subgroup
$$ \Omega=\Omega(A,B)=\rho_{(B:A)/B}\big(\Cent_{R/B}((A+B)/B)\big)
\cap (R/(B:A))^* $$
\noindent
Let $G=\GL(n,R)$ and 
\begin{multline*}
C_{\Omega(A,B)}(n,R,(B:A))=\{g\in\GL(n,R)\mid g_{ij}, g_{ii}-g_{jj}\in (B:A)\\ 
\text{\ for\ } i\neq j, \text{\ and\ } g_{ii}\in\Cent_R(A,B)\}. 
\end{multline*}
\noindent
In other words, this group is defined in exactly the same way as the 
full congruence subgroup $C(n,R,(B:A))$, only that now instead of 
requiring that the diagonal entries of matrices become central modulo 
$(B:A)$, we impose a weaker condition that modulo $B$ they commute 
with elements of $A$. Of course, since $(B:A)\circ A\le B$, this condition
depends not on the entry itself, but only on its congruence class modulo
$(B:A)$, which secures correctness of this definition.
\par
In particular, when $R$ is commutative, 
$C_{\Omega(A,B)}(n,R,(B:A))=C(n,R,(B:A))$
is the usual full congruence subgroup of level $(B:A)$.


\section{Centralisers of $E(n,R,A)$ and $\GL(n,R,A)$, 
modulo $\GL(n,R,B)$}

Now we are all set to prove the first main result of the present paper.
Here we consider relative centralisers modulo the principal congruence
subgroups $\GL(n,R,B)$, which are {\it always\/} normal in $\GL(n,R)$,
which makes the analysis considerable 

\begin{theorem}
Let $R$ be an arbitrary associative ring with 1, $A,B\unlhd R$, and 
$n\ge 3$. Further, let $H\le\GL(n,R)$ be any subgroup such that 
$E(n,A)\le H\le\GL(n,R,A)$. Then 
$$ C_{\GL(n,R)}(H,\GL(n,R,B)) = C_{\Omega(A,B)}(n,R,(B:A)). $$
\end{theorem}
\begin{proof}
First, we calculate $C_{\GL(n,R)}(E(n,A),\GL(n,R,B))$.
Assume that $g\in\GL(n,R)$ commutes with $E(n,A)$ modulo $\GL(n,R,B)$.
In particular this means that for all $1\le r\neq s\le n$ and all $a\in A$ one has
$[g,t_{rs}(a)]\in\GL(n,R,B)$, so that 
$$ gt_{rs}(a)\equiv t_{rs}(a)g\pamod{B}. $$
\noindent
The left hand side only differs from $g$ in the $s$-th column, while the
right hand side only differs from $g$ in the $r$-th row.
Comparing the entries of these matrices in positions $(i,s),(r,j)\neq (r,s)$, we
see that
\par\smallskip
$\bullet$ $g_{ir}a\in B$, for all $i\neq r$, 
\par\smallskip
$\bullet$ $ag_{sj}\in B$, for all $j\neq s$, 
\par\smallskip\noindent
in particular, we can conclude that all non-diagonal entries of $g$ belong to $(B:A)$.
\par
It remains to compare the entries in the position $(r,s)$. 
\par\smallskip
$\bullet$ $g_{rr}a\equiv ag_{ss}\pamod{B}$, or, what is the same,
$g_{rr}a-ag_{ss}\in B$, for all $r\neq s$.
\par\smallskip\noindent
Now, since $n\ge 3$, we can choose an index $t\neq r,s$, and conclude
that $g_{rr}a-ag_{tt}\in B$. Comparing the above inclusions, we see that
$a(g_{ss}-g_{tt})\in B$, so that $g_{ss}-g_{tt}\in A^{-1}B$, for all $s\neq t$.
By the same token, $ag_{ss}-g_{tt}a\in B$, and again comparing the above
inclusions we see that $(g_{rr}-g_{tt})a\in B$, so that 
$g_{rr}-g_{tt}\in BA^{-1}$, for all $r\neq t$. This means that pairwise
differences of the diagonal entries of $g$ belong to $(B:A)$. But then the 
congruence in the last item implies that $g_{rr}a-ag_{rr}\in B$, for all $r$,
$1\le r\le n$, and all $a\in A$. Summarising the above, we see that
$$ C_{\GL(n,R)}(E(n,A),\GL(n,R,B))\le C_{\Omega(A,B)}(n,R,(B:A)). $$
\par
On the other hand,
\begin{multline*}
C_{\GL(n,R)}(\GL(n,R,A),\GL(n,R,B))\le C_{\GL(n,R)}(H,\GL(n,R,B))\le\\ 
C_{\GL(n,R)}(E(n,A),\GL(n,R,B)), 
\end{multline*}
\noindent
and to prove the theorem it only remains to verify that
$$ C_{\Omega(A,B)}(n,R,(B:A))\le C_{\GL(n,R)}(\GL(n,R,A),\GL(n,R,B)). $$
\noindent
Indeed, let $g=(g_{ij})\in C_{\Omega(A,B)}(n,R,(B:A))$ and $h\in\GL(n,R,A)$. We
claim that then $gh\equiv hg\pamod{B}$. Indeed, 
$$ (gh)_{rs}=\sum g_{rt}h_{ts},\qquad (gh)_{rs}=\sum g_{rt}h_{ts}, $$
\noindent
where both sums are taken over $1\le t\le n$. 
\par\smallskip
$\bullet$ First, let $r\neq s$. Then the summands corresponding to
$t\neq r,s$ belong to $(B:A)A\le  B$ on the left hand side, and to $A(B:A)\le B$
on the right hand side, and can be discarded. 
\par\smallskip
$\bullet$ This means that for the case $r\neq s$ it only remains to take care of the
summands corresponding to $t=r,s$. But since $r\neq s$ one has $h_{rs}\in A$,
so that $g_{rr}h_{rs}\equiv h_{rs}g_{ss}\pamod{B}$ by the very definition
of $C_{\Omega(A,B)}(n,R,(B:A))$. On the other hand,
$h_{rr}\equiv h_{ss}\equiv 1\pamod{A}$ and since $g_{rs}\in (B:A)$, also
$g_{rs}h_{ss}\equiv h_{rr}g_{rs}\pamod{B}$.
\par\smallskip
$\bullet$ By the same token, for the remaining case $r=s$ all summands 
$g_{rt}h_{tr}$ and $h_{rt}g_{tr}$ belong to $B$ and can be discarded.
On the other hand, $h_{rr}\equiv 1\pamod{A}$, and since 1 commutes with
$g_{rr}$, while elements of $A$ commute with $g_{rr}$ modulo $B$, one
has $g_{rr}h_{rr}\equiv h_{rr}g_{rr}\pamod{B}$, as claimed.
\par\smallskip
This proves the desired inclusion, and thus the theorem.
\end{proof}


\section{Generation of relative subgroups and commutator formulas}

In the present section we collect the requisite results on relative
elementary subgroups that will be used in the rest of this paper.

The following lemma on generation of relative elementary subgroups
$E(n,R,A)$ is a classical result discovered in various contexts by 
Stein, Tits and Vaserstein, see, for instance, \cite{Vaserstein_normal}
(or \cite{Hazrat_Vavilov_Zhang}, Lemma 3, for a complete elementary
proof). It is stated in terms of the {\it Stein---Tits---Vaserstein 
generators\/}):
$$ z_{ij}(a,c)=t_{ij}(c)t_{ji}(a)t_{ij}(-c),\qquad
1\le i\neq j\le n,\quad a\in A,\quad c\in R. $$

\begin{lemma}
Let $R$ be an associative ring with $1$, $n\ge 3$, and let $A$ 
be a two-sided ideal of $R$. Then as a subgroup $E(n,R,A)$ is 
generated by $z_{ij}(a,c)$, for all $1\le i\neq j\le n$, $a\in A$, 
$c\in R$.
\end{lemma}

In the following theorem a further type of generators occur, the
{\it elementary commutators\/}:
$$ y_{ij}(a,b)=[t_{ij}(a),t_{ji}(b)],\qquad
1\le i\neq j\le n,\quad a\in A,\quad b\in B. $$

The following analogue of Lemma~1 for commutators 
$[E(n,R,A),E(n,R,B)]$ was discovered (in slightly less precise 
forms) by Roozbeh Hazrat and the second author, see 
\cite{Hazrat_Zhang_multiple}, Lemma 12 and then in our joint 
paper with Hazrat \cite{Hazrat_Vavilov_Zhang}, Theorem~3A. 
The strong form reproduced below is established only in our
paper \cite{NZ2}, Theorem~1 (see also \cite{NZ3}), as a 
spin-off of our papers \cite{NV18, NZ1}.

\begin{lemma}
Let $R$ be any associative ring with $1$, let $n\ge 3$, and let $A,B$ 
be two-sided ideals of $R$. Then the mixed commutator subgroup 
$[E(n,R,A),E(n,R,B)]$ is generated as a group by the elements of the form
\par\smallskip
$\bullet$ $z_{ij}(ab,c)$ and $z_{ij}(ba,c)$,
\par\smallskip
$\bullet$ $y_{ij}(a,b)$,
\par\smallskip\noindent
where $1\le i\neq j\le n$, $a\in A$, $b\in B$, $c\in R$. Moreover,
for the second type of generators, it suffices to fix one pair of
indices $(i,j)$.
\end{lemma}

In the proofs below we use not just Lemma 2 itself, but also some 
of the results used in its proof. The first of them is standard, see,
for instance, \cite{NVAS,NVAS2,Hazrat_Vavilov_Zhang} and
references there.

\begin{lemma}
$R$ be an associative ring with $1$, $n\ge 3$, and let $A$ and $B$
be two-sided ideals of $R$.  Then 
$$ E(n,R,A\circ B)\le\big[E(n,A),E(n,B)\big]\le
\big[E(n,R,A),E(n,R,B)\big] \le\GL(n,R,A\circ B). $$
\end{lemma}

The first of the following lemmas is \cite{NZ2}, Lemma 3, or \cite{NZ3}, 
Lemma 9. The second is \cite{NZ3}, Lemma 10. And the third one is
\cite{NZ2}, Lemma 5, or \cite{NZ3}, Lemma 11. 

\begin{lemma}
Let $R$ be an associative ring with $1$, $n\ge 3$, and let $A,B$ 
be two-sided ideals of $R$. Then for any  $1\le i\neq j\le n$, 
$a\in A$, $b\in B$, and any $x\in E(n,R)$ one has
$$ {}^x y_{ij}(a,b)\equiv  y_{ij}(a,b) \pamod{E(n,R,A\circ B)}. $$
\end{lemma}

\begin{lemma}
Let $R$ be an associative ring with $1$, $n\ge 3$, and let $A,B$ 
be two-sided ideals of $R$. Then for any  $1\le i\neq j\le n$, 
$a,a_1,a_2\in A$, $b,b_1,b_2\in B$ one has
\begin{align*}
&y_{ij}(a_1+a_2,b)\equiv  y_{ij}(a_1,b)\cdot y_{ij}(a_1,b) 
\pamod{E(n,R,A\circ B)},\\
&y_{ij}(a,b_1+b_2)\equiv  y_{ij}(a,b_1)\cdot y_{ij}(a,b_2) 
\pamod{E(n,R,A\circ B)},\\
&y_{ij}(a,b)^{-1}\equiv  y_{ij}(-a,b)\equiv y_{ij}(a,-b) 
\pamod{E(n,R,A\circ B)},\\
&y_{ij}(ab_1,b_2)\equiv y_{ij}(a_1,a_2b)\equiv e
\pamod{E(n,R,A\circ B)}.
\end{align*}
\end{lemma}

\begin{lemma}
Let $R$ be an associative ring with $1$, $n\ge 3$, and let $A,B$ 
be two-sided ideals of $R$. Then for any  $1\le i\neq j\le n$, any
$1\le k\neq l\le n$, and all $a\in A$, $b\in B$, $c\in R$, one has
$$ y_{ij}(ac,b)\equiv  y_{kl}(a,cb) \pamod{E(n,R,A\circ B)}. $$
\end{lemma}

For quasi-finite rings the following result is \cite{NVAS2},
Theorem 5 and \cite{Hazrat_Vavilov_Zhang}, Theorem 2A, 
but for arbitrary associative rings it was only established in 
\cite{NZ3}, Theorem 2.

\begin{lemma}
Let $R$ be any associative ring with $1$, let $n\ge 3$, and let 
$A$ and $B$ be two-sided ideals of $R$. If $A$ and $B$ are
comaximal, $A+B=R$, then
$$ [E(n,A),E(n,B)]=E(n,R,A\circ B). $$
\end{lemma}

The following result is \cite{NVAS2}, Theorem 4. 

\begin{lemma}
Let $A$ and $B$ be two ideals of a commutative 
ring $R$ and  $n\ge 3$. Then 
$$ \big[E(n,R,A),C(n, R, B)\big]=\big[E(n,R,A),E(n,R,B)\big], $$
\end{lemma}

Finally, the following lemma is \cite{NV19}, Theorem 2.

\begin{lemma} 
Let $A$ and $B$ be two ideals of a Dedekind ring of arithmetic 
type $R={\mathcal O}_S$. Assume that the multiplicative 
group $R^*$ is infinite and that $n\ge 3$. Then
$$ \big[\GL(n,R,A),\GL(n,R,B)\big]=E(n,R,AB). $$
\end{lemma}


\section{Centralisers of $E(n,R,A)$ and $\GL(n,R,A)$, 
modulo $E(n,R,B)$}

The group $Z(n,R,I)$ is defined as the centraliser of $\GL(n,R)$ modulo
$E(n,R,I)$:
$$ Z(n,R,I)=\big\{g\in\GL(n,R) \mid [g,\GL(n,R)]\le E(n,R,I) \big\}. $$
\noindent
When $E(n,R,I)$ is normal in $\GL(n,R)$, the quotient
$Z(n,R,I)/E(n,R,I)$ is the centre of $\GL(n,R)/E(n,R,I)$. 
\par
Let us make some 
obvious observations concerning this group.
\par\smallskip
$\bullet$ By definition
$$ C(n,R,I)=\big\{g\in\GL(n,R) \mid [g,\GL(n,R)]\le\GL(n,R,I) \big\}. $$
\noindent
In other words, $C(n,R,I)/\GL(n,R,I)$ is the centre of $\GL(n,R)/\GL(n,R,I)$.
Since $E(n,R,I)\le\GL(n,R,I)$, one has $Z(n,R,I)\le C(n,R,I)$. 
\par\smallskip
$\bullet$ Since $[\GL(n,R,I),\GL(n,R)]\le E(n,R,I)$ for $n\ge\max(\sr(R)+1,3)$,
in this case $E(n,R,I)$ is normal in $\GL(n,R)$ and 
$\GL(n,R,I)/E(n,R,I)$ is contained in the centre of $\GL(n,R)/E(n,R,I)$. Thus,
in the stable range
$$ \GL(n,R,I)\le Z(n,R,I)\le C(n,R,I). $$
\noindent
However even in the stable range, it may happen that $Z(n,R,I)$ is strictly
smallser than $C(n,R,I)$. 
\par\smallskip
$\bullet$ Below the stable range funny things may happen. In particular,
below the stable range even for commutative rings and $n\ge 3$ the group
$$ K_1(n,R,I)=\GL(n,R,I)/E(n,R,I) $$
\noindent
does not have to be abelian. The first such counter-examples were constructed 
by Wilberd van der Kallen \cite{vdK_module_structure} and Anthony Bak 
\cite{Bak_nonabelian}. For finite dimensional rings this group is indeed 
nilpotent by abelian, but the nilpotent part may have arbitrarily large
nilpotency class. 

This means that $Z(n,R,I)$ may sit at the very bottom of $\GL(n,R,I)$. 
Both Alec Mason \cite{Mason_nonnormal} and Anthony Bak 
\cite{Bak_nonabelian}
used the fact that $Z(n,R,I)<C(n,R,I)$ to construct subgroups of
level $I$ that are normalised by $E(n,R)$, but not normal in $\GL(n,R)$.

This means that in general such relative centralisers as 
$$  C_{\GL(n,R)}(\GL(n,R,A),E(n,R,B))\quad \text{and}\quad C_{\GL(n,R)}(E(n,R,A),E(n,R,B)) $$
\noindent
do not have an obvious description in the style of the previous 
section. But for commutative rings it is very easy to slightly 
modify $E(n,R,B)$, to get exactly the same answer, as above.
In this case all commutativity assumptions are automatically satisfied, 
so that $C(n,R,I)=C^*(n,R,I)$ for all ideals.

\begin{theorem}
Let $R$ be a commutative ring and $A,B\unlhd R$, $n\ge 3$. Then 
$$ C_{\GL(n,R)}(E(n,R,A),[E(n,R,(B:A)),E(n,R,A)]) = C(n,R,(B:A)). $$
\end{theorem}
\begin{proof}
By Lemma~8 one has
$$ [C(n,R,(B:A)),E(n,R,A)]=[E(n,R,(B:A)),E(n,R,A)]. $$
\noindent
In other words, the right hand side of the equality in the statement
of theorem is contained in the left hand side.
\par
On the other hand, 
\begin{multline*}
[E(n,R,(B:A)),E(n,R,A)]\le [\GL(n,R,(B:A)),\GL(n,R,A)]\le\\ \GL(n,R,(B:A)A),
\end{multline*}
\noindent
and thus, by Theorem 1
\begin{multline*}
C_{\GL(n,R)}(E(n,R,A),[E(n,R,(B:A)),E(n,R,A)])\le\\ 
C_{\GL(n,R)}(E(n,R,A),\GL(n,R,(B:A)A))=C(n,R,((B:A)A:A)).
\end{multline*}
\par
It remains only to observe that the level of this last subgroup is precisely the
what it should be, $((B:A)A:A)=(B:A)$. Indeed, $x\in((B:A)A:A)$
means that $xA\le (B:A)A\le B$, so that $((B:A)A:A)\le (B:A)$. On the other 
hand, if $x\in (B:A)$, then $xA\le (B:A)A$, so that $(B:A)\le ((B:A)A:A)$.
This means that $C(n,R,((B:A)A:A))= C(n,R,(B:A))$, which proves the
theorem.
\end{proof}

Of course, the level of the commutator subgroup 
$[E(n,R,(B:A)),E(n,R,A)]$ is $(B:A)A$, which in general is {\it smaller\/} 
than $B$. However, for Dedekind rings always $(B:A)A=B$.
\par
On the other hand, when the levels coincide, the mixed commutator subgroup
$[E(n,R,(B:A)),E(n,R,A)]$ is in general {\it larger\/} than $E(n,R,B)$.
It suffices to take the known examples, where $A=(B:A)=I$, while $B=I^2$,
see \cite{NZ3}. The simplest such example was constructed by Alec Mason and
Wilson Stothers \cite{Mason_Stothers, Mason_commutators_2} already for
the ring $\Int[i]$ of Gaussian integers. However, in \cite{NV19} the first author 
noticed that this cannot possibly occur for Dedekind rings of arithmetic type with 
{\it infinite\/} multiplicative group.

\begin{theorem}
Let $R$ be a Dedekind ring of arithmetic type with infinite multiplicative
group and $A,B\unlhd R$, $n\ge 3$. Then 
$$ C_{\GL(n,R)}(E(n,R,A),E(n,R,B)) = C(n,R,(B:A)). $$
\end{theorem}
\begin{proof}
Indeed, by Lemma 9, one has
$$ [E(n,R,(B:A)),E(n,R,A)]=E(n,R,(B:A)A)=E(n,R,B), $$
\noindent
and it remains to apply the previous theorem.
\end{proof}


\section{Final remarks}

It would be natural to generalise results of the present paper
to more general contexts.

\begin{problem}
Generalise Theorems $1$ and $2$ to Chevalley groups.
\end{problem}

\begin{problem}
Generalise Theorems $1$ and $2$ to Bak's unitary groups.
\end{problem}

It seems, that in both cases the strategy is clear, but there are a lot 
of technical details to take care of. For Chevalley groups the ground
ring $R$ is commutative anyway, which makes many technical details
considerably less burdensome. On the other hand, calculations
in representations themselves necessary to establish analogues of 
Theorem 1, will be somewhat more delicate, especially for exceptional 
groups. But the pattern of such calculations
should be mostly known from \cite{VG,VLE6,VLE7,VN}. Most tools
necessary to derive from there an analogue of Theorem 2, are in our
recent papers \cite{NZ1,NZ4}. On the other hand, so far we are still
missing several key components necessary to generalise to this case 
the results of \cite{NV19}, and before we do that, there is no hope 
to generalise Theorem 3.

The authors thank Roozbeh Hazrat and Alexei Stepanov for ongoing 
discussion of this circle of ideas, and long-standing cooperation 
over the last decades. 


\frenchspacing
\par
\medskip

\providecommand{\bysame}{\leavevmode\hbox to3em{\hrulefill}\thinspace}
\providecommand{\MR}{\relax\ifhmode\unskip\space\fi MR }
\providecommand{\MRhref}[2]{%
  \href{http://www.ams.org/mathscinet-getitem?mr=#1}{#2}
}
\providecommand{\href}[2]{#2}

\end{document}